\documentclass{amsart}

\usepackage[latin1]{inputenc}
\usepackage[english]{babel}
\usepackage{indentfirst}
\usepackage{amssymb}
\usepackage{amsthm}
\usepackage{xcolor}
\usepackage[all]{xy}
\usepackage[mathscr]{eucal}
\usepackage{hyperref}

\newcommand{\N}{\mathbb{N}}
\newcommand{\erre}{\mathbb{R}}
\newcommand{\sub}{\subseteq}
\def\epsilon{\varepsilon}

\newcommand{\diam}{\mathop{\mathrm{diam}}\nolimits}

\newcommand{\pten}{\ensuremath{\widehat{\otimes}_\pi}}
\DeclareMathOperator{\conv}{conv}

\newtheorem{theorem}{Theorem}[section]
\newtheorem{proposition}[theorem]{Proposition}
\newtheorem{corollary}[theorem]{Corollary}
\newtheorem{lemma}[theorem]{Lemma}

\theoremstyle{definition}

\newtheorem{remark}[theorem]{Remark}

\numberwithin{equation}{section}

\title{On weakly almost square Banach spaces}

\author[J. Rodr\'iguez]{Jos\'e Rodr\'{i}guez}
\address[J. Rodr\'iguez]{Departamento de Ingenier\'{i}a y Tecnolog\'{i}a de Computadores,
Facultad de Inform\'{a}tica, Universidad de Murcia, 30100 Espinardo (Murcia), Spain.
	\newline
	\href{https://orcid.org/0000-0001-5316-8016}{ORCID: \texttt{0000-0001-5316-8016} }}
\email{joserr@um.es}
\urladdr{\url{https://webs.um.es/joserr}}

\author[A. Rueda Zoca]{Abraham Rueda Zoca}
\address[A. Rueda Zoca]{Departamento de An\'{a}lisis Matem\'{a}tico, Facultad de Ciencias, Universidad de Granada, 18071 Granada, Spain.
	\newline
	\href{https://orcid.org/0000-0003-0718-1353}{ORCID: \texttt{0000-0003-0718-1353} }}
\email{abrahamrueda@ugr.es}
\urladdr{\url{https://arzenglish.wordpress.com}}



\thanks{The research was supported by grants PID2021-122126NB-C32 (J. Rodr\'{i}guez) and 
PID2021-122126NB-C31 (A. Rueda Zoca), funded by MCIN/AEI/10.13039/501100011033 and ``ERDF A way of making Europe'', 
and also by grant 21955/PI/22 (funded by Fundaci\'on S\'eneca - ACyT Regi\'{o}n de Murcia). 
The research of A. Rueda Zoca was also supported by grants FQM-0185 and PY20\_00255 (funded by Junta de Andaluc\'ia).}

\begin{document}

\begin{abstract}
We prove some results on weakly almost square Banach spaces and their relatives. 
On the one hand, we discuss weak almost squareness in the setting of Banach function spaces. More precisely, 
let $(\Omega,\Sigma)$ be a measurable space, let $E$ be a Banach lattice and let $\nu:\Sigma \to E^+$ be a non-atomic countably additive measure
having relatively norm compact range. Then the space $L_1(\nu)$ is weakly almost square. This result
applies to some abstract Ces\`{a}ro function spaces. Similar arguments show that
the Lebesgue-Bochner space $L_1(\mu,Y)$ is weakly almost square for any Banach space~$Y$ and for any non-atomic finite measure~$\mu$. 
On the other hand, we make some progress on the open question of whether there exists a locally almost square Banach space which fails the diameter two property. 
In this line we prove that if $X$ is any Banach space containing a complemented isomorphic copy of~$c_0$, then for every $0<\varepsilon<1$ there exists an equivalent norm 
$|\cdot|$ on~$X$ satisfying: (i)~every slice of the unit ball $B_{(X,|\cdot|)}$ has diameter~$2$; (ii) $B_{(X,|\cdot|)}$ contains non-empty relatively weakly open subsets of arbitrarily small diameter; and (iii)~$(X,|\cdot|)$ is $(r,s)$-SQ for all $0<r,s < \frac{1-\varepsilon}{1+\varepsilon}$.
\end{abstract}

\maketitle

\section{Introduction}

Let $(X,\|\cdot\|)$ be a Banach space. 
The (closed) unit ball and the unit sphere of~$X$ are denoted by $B_{(X,\|\cdot\|)}$ and $S_{(X,\|\cdot\|)}$, respectively. If the norm does not 
need to be explicitly mentioned, we just write $B_X$ and~$S_X$ instead.
Given a bounded set $C \sub X$, 
a \textit{slice} of $C$ is a set of the form
$$
	S(C,x^*,\alpha):=\{x\in C: \, x^*(x)>\sup x^*(C)-\alpha\}
$$
for some $x^*\in X^*$ (the topological dual of $X$) and $\alpha>0$. Notice that $S(C,x^*,\alpha)$ is non-empty and relatively weakly open in~$C$. 
A Banach space is said to have the \textit{slice diameter two property (slice-D2P)} (respectively, \textit{diameter two property -- D2P, strong diameter two property -- SD2P}) if every slice (respectively, non-empty relatively weakly open subset, convex combination of slices) of the unit ball has diameter~$2$. 
Diameter two properties have attracted the attention of many researchers in the last 20 years (see, e.g., \cite{abr-alt-2,abr-lim-nyg,bec-lop-rue:15a,bec-lop-rue:15b,hal-lan-pol}) and 
have motivated the appearance of new properties of Banach spaces (almost squareness \cite{abr-lan-lim}, symmetric strong diameter two properties \cite{hal-alt, lan-rue} or diametral diameter two properties \cite{bec-lop-rue:18}). 

According to Abrahamsen, Langemets and Lima~\cite{abr-lan-lim}, a Banach space $(X,\|\cdot\|)$ is said to be 
\begin{enumerate}
\item[(i)] \textit{locally almost square (LASQ)} if for every $x\in S_{(X,\|\cdot\|)}$ there exists a sequence 
$(y_n)_{n\in \N}$ in~$B_{(X,\|\cdot\|)}$ such that $\Vert x\pm y_n\Vert\rightarrow 1$ and $\Vert y_n\Vert\rightarrow 1$;
\item[(ii)] \textit{weakly almost square (WASQ)} if for every $x\in S_{(X,\|\cdot\|)}$ there exists a weakly null sequence 
$(y_n)_{n\in \N}$ in~$B_{(X,\|\cdot\|)}$ such that $\Vert x\pm y_n\Vert\rightarrow 1$ and 
$\Vert y_n\Vert\rightarrow 1$;
\item[(iii)] \textit{almost square (ASQ)} if for every finite set $\{x_1,\ldots, x_k\} \sub S_{(X,\|\cdot\|)}$ there exists a sequence 
$(y_n)_{n\in \N}$ in~$B_{(X,\|\cdot\|)}$ 
such that $\Vert x_i\pm y_n\Vert\rightarrow 1$ for every $i\in\{1,\ldots, k\}$ and $\Vert y_n\Vert\rightarrow 1$.
\end{enumerate} 
All these properties are isometric in nature, i.e., they depend on the norm considered. For instance, the basic example of an ASQ space is $c_0$ with its usual norm, while
every Banach space admits an equivalent norm failing the slice-D2P and so it cannot be LASQ (see, e.g., \cite[Lemma~2.1]{bec-lop-rue:16b}).

Even though the previous properties were introduced in \cite{abr-lan-lim}, 
LASQ and WASQ spaces were implicitly used by Kubiak~\cite{kub} to study the D2P in some Ces\`aro function spaces.
Apart from being interesting by themselves, almost squareness properties have shown to be a powerful tool in order to study 
diameter two properties in certain Banach spaces where there is not a good description of the dual space. In this direction let us mention, for instance, that 
in \cite[Section~4]{har-2} it is proved that if $X$ is LASQ (resp., ASQ) then any ultrapower $X_\mathcal U$ of~$X$ is LASQ (resp., ASQ) and, in particular, $X_\mathcal U$ has the slice-D2P (resp., SD2P). Observe that it is unclear whether $X_\mathcal U$ has the slice-D2P (resp., SD2P) if $X$ has the slice-D2P (resp., SD2P). 
Another context in which these properties are useful are the projective symmetric tensor products. It is known that if $X$ is WASQ and has the Dunford-Pettis property (resp., $X$ is ASQ) then all the projective symmetric tensor products $\widehat{\otimes}_{\pi, s, N}X$ have the slice-D2P \cite[Proposition 3.6]{lan-lim-rue} (resp., SD2P \cite[Theorem 3.3]{bec-lop-rue:16}). 
Notice that it is unknown whether any of the diameter two properties is stable by taking projective symmetric tensor products. 

Among all the almost squareness properties introduced in~\cite{abr-lan-lim}, it is clear that ASQ has been studied in a more intensive way 
because it turns out to characterise the containment of~$c_0$. More precisely, a Banach space admits an ASQ equivalent renorming if, and only if, it contains an isomorphic 
copy of~$c_0$ (see \cite[Lemma~2.6]{abr-lan-lim} and \cite[Theorem~2.3]{bec-lop-rue:16}). 
The contribution to examples of LASQ and WASQ spaces has been more modest. In spite of that, we find several results in the literature about these properties in the context of function spaces. On the one hand, Kubiak proved in \cite[Lemma~3.3]{kub} that the weighted Ces\`aro function spaces on an interval are WASQ. 
In particular, $L_1[0,1]$ is WASQ. On the other hand, Hardtke proved in \cite[Theorem~3.1]{har-3} that the K\"othe-Bochner space $E(X)$ is LASQ whenever the Banach space $X$ is LASQ, for any Banach function space~$E$.

The aim of this note is to deepen the understanding of WASQ and LASQ Banach spaces. The paper is organized as follows.

In Section~\ref{section:functionspaces} we focus on certain Banach function spaces which play an important role in Banach lattice and operator theory. Namely,
we consider the space $L_1(\nu)$ of all real-valued functions that are integrable with respect to a countably additive vector measure $\nu$ (defined on a $\sigma$-algebra and
taking values in a Banach space). Up to Banach lattice isometries, these spaces represent all order continuous Banach lattices having a weak order unit (see, e.g., \cite[Theorem~8]{cur1}). 
Therefore, there are reflexive (hence, having the Radon-Nikod\'{y}m property and so failing the slice-D2P) Banach lattices within this class, like $\ell_p$ and $L_p[0,1]$ for $1<p<\infty$. For detailed information on the $L_1$ space of a vector measure, see~\cite{oka-alt}. More recent references on this topic
are \cite{cal-alt-6,cur-ric-6,cur-ric-5,nyg-rod,rod16}. Our main result in this section is the following:

\begin{theorem}\label{theorem:main}
Let $(\Omega,\Sigma)$ be a measurable space, let $E$ be a Banach lattice and let $\nu:\Sigma \to E$ be a countably additive measure. If $\nu$ is non-atomic
and the set 
$$
	\mathcal{R}(\nu):=\{\nu(A): \, A\in \Sigma\}
$$ 
(the range of~$\nu$) is a relatively norm compact subset of~$E^+:=\{x\in E: x\geq 0\}$, then $L_1(\nu)$ is WASQ. 
\end{theorem}

Clearly, Theorem~\ref{theorem:main} generalizes the fact that the classical space $L_1(\mu)$ of a non-atomic finite measure~$\mu$ is WASQ.
As an application of Theorem~\ref{theorem:main} and some results of Curbera and Ricker~\cite{cur-ric-4}, it follows that if $E$ is an order continuous 
rearrangement invariant Banach function space on~$[0,1]$, 
then the abstract Ces\`aro function space $[\mathcal C,E]$ is WASQ (Corollary \ref{corollary:Cesaro}).  This generalizes 
the aforementioned result by Kubiak in the case of the interval~$[0,1]$. Abstract Ces\`aro function spaces
have been widely studied in the literature (see, e.g., \cite{ast-les-mal,ast-mal,cur-ric-4}).

The techniques of Theorem~\ref{theorem:main} allow us to show the Lebesgue-Bochner space
$L_1(\mu,Y)$ is WASQ for any Banach space~$Y$ whenever $\mu$ is a non-atomic finite measure (Corollary~\ref{corollary:Lebesgue-Bochner}).
This result should be compared with the above mentioned result of~\cite{har-3} that the property of being LASQ passes from a Banach space~$Y$ to the K\"othe-Bochner space
$E(Y)$, for any Banach function space~$E$. We finish Section~\ref{section:functionspaces} with an example of a WASQ Banach space of the form $L_1(\nu)$
as in Theorem~\ref{theorem:main} which is not an $\mathcal L_1$-space (Subsection~\ref{subsection:nakano}).

In Section~\ref{section:renorming} we go a bit further in the analysis of the link between almost squareness and diameter two properties.
One of the main questions raised in~\cite{abr-lan-lim} is whether there exists a LASQ Banach space which is not WASQ. Very recently, Kaasik and Veeorg proved 
in~\cite[Section~2]{kaa-vee} that the answer is negative and that an example can be found in the class of Lipschitz-free spaces over complete metric spaces. For such spaces,
the properties SD2P, D2P, slice-D2P and LASQ are equivalent (combine \cite[Theorem~1.5]{avi-mar-2} and 
\cite[Theorem~3.1]{hal-kaa-ost}), so the above mentioned example satisfies the SD2P. Since the slice-D2P and the D2P are different properties~\cite{bec-lop-rue:15a}, it is a natural question whether there exists a LASQ Banach space which fails the D2P.  Within the framework of Banach lattices, 
a stronger version of the LASQ property which implies the D2P has been considered in~\cite{cia}. Even though 
we do not know the answer to the previous question, we make some progress 
in this direction. Our main result in Section~\ref{section:renorming} is the following:
 
\begin{theorem}\label{theorem:countere}
Let $X$ be a Banach space containing a complemented isomorphic copy of~$c_0$. Then for any $0<\varepsilon<1$ there exists an equivalent norm $|\cdot|$ on~$X$ such that:
\begin{enumerate}
\item[(i)] $(X,|\cdot|)$ has the slice-D2P, that is, every slice of $B_{(X,|\cdot|)}$ has diameter~$2$.
\item[(ii)] There are non-empty relatively weakly open subsets of $B_{(X,|\cdot|)}$ of arbitrarily small diameter.
\item[(iii)] $(X,|\cdot|)$ is $(r,s)$-SQ for all $0<r,s < \frac{1-\varepsilon}{1+\varepsilon}$ in the sense of~\cite[Section~6]{avi-alt-7}, 
that is, for every finite set $\{x_1,\ldots, x_n\}\sub S_X$ there exists $y\in S_X$ satisfying
 \begin{equation*}
     |rx_i \pm sy|\le1\ \text{ for every }i\in \{1,\dots,n\}.
 \end{equation*}
\end{enumerate}
\end{theorem}

Condition~(iii) measures somehow how far is the norm from being ASQ. Notice that a Banach space is ASQ if and only if it is
$(r,s)$-SQ for all $0<r,s<1$.

Theorem~\ref{theorem:countere} applies to any separable Banach space containing an isomorphic copy of~$c_0$, thanks to Sobczyk's theorem.
The proof of Theorem~\ref{theorem:countere} is inspired by the renorming technique developed by Becerra Guerrero, L\'{o}pez-Perez and Rueda Zoca 
in~\cite[Theorem 2.4]{bec-lop-rue:15a}, which in turn uses ideas of the 
example of Argyros, Odell and Rosenthal~\cite{arg-ode-ros} of a closed bounded convex subset of~$c_0$ having the convex point of continuity 
property but failing the point of continuity property.

\subsection*{Terminology}
We follow standard notation as can be found in~\cite{alb-kal,die-uhl-J}. We will consider real Banach spaces only. 
By an {\em operator} we mean a continuous linear map between Banach spaces.
By a {\em subspace} of a Banach space we mean a norm closed linear subspace.
Let $(X,\|\cdot\|)$ be a Banach space. Given a set $C \sub X$, we denote by
${\rm conv}(C)$ (resp., $\overline{{\rm conv}}(C)$) its convex hull (resp., closed convex hull).
The {\em diameter} of~$C$ is defined by ${\rm diam}_{\|\cdot\|}(C):=\sup\{\|x-x'\|:x,x'\in C\}$
and will be also denoted by ${\rm diam}(C)$ if no confusion arises.
 
Let $(\Omega,\Sigma,\mu)$ be a finite measure space. 
A Banach space~$(E,\|\cdot\|)$ is said to be a \emph{Banach function space} on $(\Omega,\Sigma,\mu)$ (or just over~$\mu$) if the following conditions hold: 
\begin{enumerate}
\item[(i)] $E$ is a (not necessarily closed) linear subspace of $L_1(\mu)$;
\item[(ii)] if $f\in L_0(\mu)$ and $|f| \leq |g|$ $\mu$-a.e. for some $g \in E$, then $f \in E$ and $\|f\| \leq \|g\|$;
\item[(iii)] the characteristic function $\chi_A$ of each $A \in \Sigma$ belongs to~$E$. 
\end{enumerate}
In this case, $E$ is a Banach lattice when endowed with the $\mu$-a.e. order and the inclusion map from $E$ to~$L_1(\mu)$ is an operator.
A set $H \sub E$ is called {\em uniformly $\mu$-integrable} if for each $\epsilon>0$
there is $\delta>0$ such that $\|f \chi_A\| \leq \epsilon$ for every $f\in H$ and for every $A\in \Sigma$ with $\mu(A)\leq \delta$. Suppose now that 
$E$ is order continuous. Then every bounded uniformly $\mu$-integrable subset of~$E$ is relatively weakly compact (see, e.g., \cite[Proposition~2.39]{oka-alt}), but the converse might fail, 
in constrast to the case of the classical $L_1$ space of a finite measure (for which the Dunford-Pettis theorem ensures the equivalence). Moreover, given $f,g\in E$
with $f\leq g$, the order interval $[f,g] \sub E$ is uniformly $\mu$-integrable (see, e.g., \cite[Lemma~2.37]{oka-alt}) and weakly compact.

\section{WASQ Banach function spaces}\label{section:functionspaces}

We begin this section with some preliminaries on the $L_1$ space of a vector measure (see \cite[Chapter~3]{oka-alt} for the basics on this topic).
Let $(\Omega,\Sigma)$ be a measurable space, let $X$ be a Banach space and let $\nu:\Sigma\to X$ be a countably additive measure. 
A set $A\in \Sigma$ is said to be $\nu$-null if 
$\nu(B)=0$ for every $B\in \Sigma$ with~$B\sub A$.
The family of all $\nu$-null sets is denoted by $\mathcal{N}(\nu)$. We say that a property holds {\em $\nu$-a.e.} if
it holds on some $A\in \Sigma$ such that $\Omega \setminus A\in \mathcal{N}(\nu)$.
We say that a set $A\in \Sigma \setminus \mathcal{N}(\nu)$ is an atom of~$\nu$ if 
for every $B \in \Sigma$ with $B \sub A$ we have either $B\in \mathcal{N}(\nu)$ or $A\setminus B\in \mathcal{N}(\nu)$.
We say that $\nu$ is {\em non-atomic} if it has no atoms. By a {\em Rybakov control measure} of~$\nu$
we mean a finite measure of the form $\mu=|x_0^*\circ \nu|$ (the variation of the signed measure $x_0^*\circ \nu:\Sigma\to \mathbb R$) for some $x_0^*\in X^*$ such that 
$\mathcal{N}(\mu) = \mathcal{N}(\nu)$ (see, e.g., \cite[p.~268, Theorem~2]{die-uhl-J} for a proof of the existence of Rybakov control measures).
 
A $\Sigma$-measurable function $f:\Omega \to \mathbb{R}$ is called {\em $\nu$-integrable} if $f\in L_1(|x^*\circ \nu|)$ for all $x^*\in X^*$ and,
for each $A\in \Sigma$, there is $\int_A f \, d\nu\in X$
such that
$$
	x^*\left ( \int_A f \, d\nu \right)=\int_A f\,d(x^*\circ \nu)  	\quad\text{for all $x^*\in X^*$}.
$$
Identifying functions which coincide $\nu$-a.e., the set $L_1(\nu)$ of all (equivalence classes of) $\nu$-integrable functions  
is a Banach lattice with the $\nu$-a.e. order and the norm
$$
	\|f\|_{L_1(\nu)}:=\sup_{x^*\in B_{X^*}}\int_\Omega |f|\,d|x^*\circ \nu|.
$$
$L_1(\nu)$ is an order continuous Banach function space over any Rybakov control measure of~$\nu$.
The (norm~$1$) operator $I_\nu: L_1(\nu)\to X$ defined by
$$
	I_\nu(f):=\int_\Omega f\, d\nu
	\quad\text{for all $f\in L_1(\nu)$}
$$ 
is called the {\em integration operator} of~$\nu$. 

To provide a proof of Theorem \ref{theorem:main} we need a couple of lemmata.
The first one belongs to the folklore (cf. \cite[Lemma~6.3.2]{alb-kal} for the case of the unit interval):

\begin{lemma}\label{lemma:Rademacher-weakly-null}
Let $(\Omega,\Sigma,\mu)$ be a non-atomic finite measure space. Then there is a sequence $(r_n)_{n\in \N}$ in~$L_\infty(\mu)$ such that:
\begin{enumerate}
\item[(i)] $|r_n|=1$ for all $n\in \N$; and 
\item[(ii)] for each $f\in L_1(\mu)$ the sequence $(fr_n)_{n\in \N}$ is weakly null in~$L_1(\mu)$. 
\end{enumerate}
\end{lemma}

A sequence as in the previous lemma will be called a {\em Rademacher-type sequence} on $(\Omega,\Sigma,\mu)$.

\begin{lemma}\label{lemma:Rademacher-weakly-null-Bfs}
Let $(\Omega,\Sigma)$ be a measurable space, let $X$ be a Banach space and let $\nu:\Sigma \to X$ be a non-atomic countably additive measure.
Let $\mu$ be a Rybakov control measure of~$\nu$ and let $(r_n)_{n\in \N}$ be a Rademacher-type sequence on $(\Omega,\Sigma,\mu)$.
Then for each $f\in L_1(\nu)$ the sequence $(fr_n)_{n\in \N}$ is weakly null in $L_1(\nu)$. 
\end{lemma}
\begin{proof} Fix $f\in L_1(\nu)$ and take any $\varphi \in L_1(\nu)^*$. Since $L_1(\nu)$ is an order continuous Banach function space over~$\mu$, there is
$g\in L_1(\mu)$ such that for each $h\in L_1(\nu)$ we have $hg\in L_1(\mu)$ and 
$$
	\varphi(h)=\int_\Omega hg \, d\mu
$$ 
(see, e.g., \cite[p.~29]{lin-tza-2}).
In particular, we have $fg\in L_1(\mu)$ and so $(fgr_n)_{n\in\N}$ is weakly null in~$L_1(\mu)$. Hence,
$\varphi(f r_n)=\int_\Omega fg r_n \, d\mu \to 0$ as $n\to \infty$.
\end{proof}

\begin{proof}[Proof of Theorem \ref{theorem:main}]
The fact that $\mathcal{R}(\nu)=\{\nu(A): A\in \Sigma\} \sub E^+$ ensures that
\begin{equation}\label{equation:positive-norm}
	\|h\|_{L_1(\nu)}=\left\|\int_\Omega |h| \, d\nu\right\|_E
	\quad\text{for all $h\in L_1(\nu)$}
\end{equation}
(see, e.g., \cite[Lemma~3.13]{oka-alt}), where $\|\cdot\|_E$ denotes the norm of~$E$.
Let $\mu$ be a Rybakov control measure of~$\nu$ and let $(r_n)_{n\in \N}$ be a Rademacher-type sequence on $(\Omega,\Sigma,\mu)$.

Fix $f\in S_{L_1(\nu)}$. For each $n\in \N$ we have $fr_n\in L_1(\nu)$ and 
$$
	\|fr_n\|_{L_1(\nu)}\stackrel{\stackrel{\eqref{equation:positive-norm}}{}}{=} \left\|\int_{\Omega} |fr_n| \, d\nu \right\|_E=\left\|\int_{\Omega} |f| \, d\nu \right\|_E
	\stackrel{\stackrel{\eqref{equation:positive-norm}}{}}{=} \|f\|_{L_1(\nu)}=1.
$$
Moreover, the sequence $(fr_n)_{n\in \N}$ is weakly null in $L_1(\nu)$ (by Lemma~\ref{lemma:Rademacher-weakly-null-Bfs}).

We claim that $\|f\pm fr_n\|_{L_1(\nu)}\to 1$ as $n\to \infty$. Indeed, for each $n\in \N$ we have $1\pm r_n\geq 0$ and so
$$
	g_n^{\pm}:=|f\pm fr_n|=|f|(1\pm r_n)=|f|\pm |f|r_n.
$$ 
Therefore, both sequences $(g_n^+)_{n\in \N}$ and $(g_n^-)_{n\in \N}$ converge weakly to~$|f|$ in~$L_1(\nu)$ (by Lemma~\ref{lemma:Rademacher-weakly-null-Bfs} applied to~$|f|$).
Hence, $(I_\nu(g_n^+))_{n\in \N}$ and $(I_\nu(g_n^-))_{n\in \N}$ converge weakly to~$I_\nu(|f|)$ in~$E$, where
$I_\nu:L_1(\nu)\to E$ denotes the integration operator of~$\nu$.
Observe that each $g_n^{\pm}$ belongs to the order interval $K:=[0,2|f|] \sub L_1(\nu)$, which is
uniformly $\mu$-integrable and weakly compact.

Since $\mathcal{R}(\nu)$ is relatively norm compact, $I_\nu$ maps
every bounded, uniformly $\mu$-integrable subset of~$L_1(\nu)$ to a relatively norm compact subset of~$E$
(see, e.g., \cite[Proposition~3.56(I)]{oka-alt}). Therefore, $I_\nu(K)$ is norm compact. It follows that both sequences
$(I_\nu(g_n^+))_{n\in \N}$ and $(I_\nu(g_n^-))_{n\in \N}$ are norm convergent to $I_\nu(|f|)$, so 
$$
	\|f\pm fr_n\|_{L_1(\nu)}\stackrel{\stackrel{\eqref{equation:positive-norm}}{}}{=}
	\left\|I_\nu(g_n^{\pm}) \right\|_E \to \left\|I_\nu(|f|) \right\|_E \stackrel{\stackrel{\eqref{equation:positive-norm}}{}}{=} \|f\|_{L_1(\nu)}=1
$$
as $n\to \infty$. The proof is finished.
\end{proof}

The rest of this section is devoted to providing applications of Theorem~\ref{theorem:main}.

\subsection{Ces\`{a}ro function spaces}\label{subsection:cesaro}
The Ces\`{a}ro ``operator'' is the map $f \mapsto \mathcal{C}(f)$ defined pointwise by $\mathcal{C}(f)(x):=\frac{1}{x}\int_0^x f(t) \, dt$
for any $f\in L_1[0,1]$. Given a rearrangement invariant Banach function space~$(E,\|\cdot\|_E)$ on~$[0,1]$, 
the Ces\`{a}ro function space $[\mathcal{C},E]$
is the Banach function space on~$[0,1]$ consisting of all $f\in L_1[0,1]$ for which $\mathcal{C}(|f|)\in E$, equipped
with the norm $\|f\|_{[\mathcal{C},E]}:=\|\mathcal{C}(|f|)\|_{E}$. 

\begin{corollary}\label{corollary:Cesaro}
Let $E$ be an order continuous rearrangement invariant Banach function space on $[0,1]$.
Then the Ces\`{a}ro function space $[\mathcal{C},E]$ is WASQ.
\end{corollary}
\begin{proof}
By \cite[Theorem~2.1]{cur-ric-4}, the formula $\nu(A):=\mathcal{C}(\chi_A)$
defines an $E^+$-valued countably additive measure on the Lebesgue $\sigma$-algebra of~$[0,1]$ such that 
$\nu$ has the same null sets as the Lebesgue measure (hence it is non-atomic) and the range of~$\nu$ is relatively norm compact.
Since $E$ is order continuous, we have $[\mathcal{C},E]=L_1(\nu)$ (see \cite[Proposition~3.1]{cur-ric-4}).
The conclusion now follows from Theorem~\ref{theorem:main}. 
\end{proof}

We stress that the weigthed Ces\`aro function space $C_{p,w}$ on $[0,1]$ considered in~\cite{kub}, for $1\leq p<\infty$ and a measurable positive function~$w$, 
is equal to $[\mathcal{C},E]$ for $E=L_p((xw(x))^p \, dx)$. Thus, the previous corollary generalizes~\cite[Lemma~3.3]{kub} in the case of~$[0,1]$.

\subsection{K\"{o}the-Bochner spaces}\label{subsection:kothe}

Let $(E,\|\cdot\|_E)$ be a Banach function space on a finite measure space $(\Omega,\Sigma,\mu)$ and let $(Y,\|\cdot\|)$ be a Banach space.
The K\"{o}the-Bochner space $E(Y)$ is the Banach space of all (equivalence classes of) strongly $\mu$-measurable functions
$f:\Omega \to Y$ such that $\|f(\cdot)\|\in E$, with the norm $\|f\|_{E(Y)}:=\|\|f(\cdot)\|\|_E$. Here
$\|f(\cdot)\|:\Omega \to \erre$ is the $\mu$-measurable function given by $t \mapsto \|f(t)\|$.

The following result should be compared with \cite[Theorem 3.1]{har-3}, where it is proved  that $E(Y)$ is LASQ if $Y$ is LASQ.

\begin{theorem}\label{theorem:Kothe-Bochner}
Let $(\Omega,\Sigma)$ be a measurable space, let $X$ be a Banach lattice and let $\nu:\Sigma \to X$ be a non-atomic countably additive measure
such that $\mathcal{R}(\nu)$ is a relatively norm compact subset of~$X^+$. Let $\mu$ be a Rybakov control measure of~$\nu$ and consider $E:=L_1(\nu)$ as 
a Banach function space on $(\Omega,\Sigma,\mu)$. Let $Y$ be a Banach space. Then the K\"{o}the-Bochner space
$E(Y)$ is WASQ.
\end{theorem}
\begin{proof}
Let $(r_n)_{n\in \N}$ be a Rademacher-type sequence on $(\Omega,\Sigma,\mu)$ (see Lemma~\ref{lemma:Rademacher-weakly-null}).

Fix $f\in S_{E(Y)}$. Then $fr_n\in S_{E(Y)}$ for every $n\in \N$ and 
we claim that $(fr_n)_{n\in \N}$ is weakly null in~$E(Y)$. Indeed, given any $\varphi \in E(Y)^*$,
the order continuity of~$E$ allows us to represent $\varphi$ as a $w^*$-scalarly $\mu$-measurable function $\varphi:\Omega \to Y^*$
such that $\|\varphi(\cdot)\| \in E^*$, the duality being given by
$$
	\varphi(h)=\int_\Omega \langle \varphi,h\rangle \, d\mu
	\quad\text{for every $h\in E(Y)$} 
$$
(see, e.g., \cite[Theorem~3.2.4]{lin-J}).
Here we denote by $\|\cdot\|$ the norm of both $Y$ and~$Y^*$, while 
$\langle \varphi,h \rangle \in L_1(\mu)$ is defined by $t \mapsto \langle \varphi(t),h(t) \rangle$.
Therefore, we have $\varphi(fr_n)=\int_\Omega \langle \varphi, f \rangle r_n \, d\mu \to 0$ as $n\to \infty$.
This shows that $(fr_n)_{n\in \N}$ is weakly null in~$E(Y)$, as claimed.

Moreover, we have $\|f(\cdot)\| \in S_E$ and so
$$
	\|f\pm fr_n\|_{E(Y)}=
	\big\|\|f(\cdot )\| (1\pm r_n)\big\|_{L_1(\nu)} \to 1
$$
as $n\to \infty$, by the proof of Theorem~\ref{theorem:main}.
\end{proof}

To the best of our knowledge, the following corollary seems to be new:

\begin{corollary}\label{corollary:Lebesgue-Bochner}
Let $(\Omega,\Sigma,\mu)$ be a non-atomic finite measure space and let $Y$ be a Banach space. Then the Lebesgue-Bochner space $L_1(\mu,Y)$
is WASQ.
\end{corollary}

\begin{remark}
The previous result is also interesting from the point of view of the identification of the space $L_1(\mu,Y)$ as the projective tensor product $L_1(\mu)\pten Y$ 
(see, e.g., \cite[p.~228, Example~10]{die-uhl-J}). In general, given two Banach spaces $X$ and~$Y$, it is not known whether $X\pten Y$ is 
WASQ if $X$ is WASQ. It is even open if $X\pten Y$ has the D2P if $X$ has the D2P (see \cite[Question~4.2]{lan-lim-rue}).
\end{remark}

\subsection{An example}\label{subsection:nakano}

The aim of this subsection 
is to give an example of a WASQ Banach function space as in Theorem~\ref{theorem:main} which is not an~$\mathcal L_1$-space. To do so, we need to introduce some terminology first.
Throughout this subsection $(p_n)_{n\in \N}$ is a sequence in~$(1,\infty)$. 

The {\em Nakano sequence space} $\ell_{(p_n)}$ is the Banach lattice
consisting of all sequences $(a_n)_{n\in \N}\in \mathbb{R}^\N$ such that $\sum_{n\in \N}|s a_n|^{p_n}<\infty$ for some $s>0$, equipped
with the coordinate-wise ordering and the norm
$$
	\big\|(a_n)_{n\in \N}\big\|_{\ell_{(p_n)}}:=\inf\left\{t>0:\, \sum_{n\in \N}\left|\frac{a_n}{t}\right|^{p_n} \leq 1 \right\}.
$$

Given a sequence of Banach spaces $(X_n,\|\cdot\|_{X_n})_{n\in \N}$, its {\em $\ell_{(p_n)}$-sum} 
is the Banach space $\ell_{(p_n)}(X_n)$ consisting of all sequences $(x_n)_{n\in \N}\in \prod_{n\in \N}X_n$ 
such that
$(\|x_n\|_{X_n})_{n\in \N}\in \ell_{(p_n)}$, with the norm
$$
	\big\|(x_n)_{n\in \N}\big\|_{\ell_{(p_n)}(X_n)}:=\Big\|\big(\|x_n\|_{X_n}\big)_{n\in \N}\Big\|_{\ell_{(p_n)}}.
$$
If $(p_n)_{n\in \N}$ is bounded, then the unit vectors form an unconditional basis of~$\ell_{(p_n)}$ (see, e.g., \cite[Theorem~3.5]{woo})
and so \cite[Proposition~5.2]{abr-lan-lim} applies to get:

\begin{corollary}\label{corollary:Nakano-sum} 
Suppose that $(p_n)_{n\in \N}$ is bounded and let $(X_n)_{n\in \N}$ be a sequence of Banach spaces which are WASQ.
Then $\ell_{(p_n)}(X_n)$ is WASQ.
\end{corollary}

We denote by $\lambda$ the Lebesgue measure on the Lebesgue $\sigma$-algebra $\Sigma$ of~$[0,1]$.

\begin{proposition}\label{proposition:example-L1nu}
Let $(A_n)_{n\in \N}$ be a partition of~$[0,1]$ such that $A_n\in \Sigma \setminus \mathcal{N}(\lambda)$ for all $n\in \N$. Then
the map $\nu:\Sigma \to \ell_{(p_n)}$ given by
$$
	\nu(A):=\big(\lambda(A\cap A_n)\big)_{n\in \N}
	\quad
	\text{for all $A\in \Sigma$}
$$
is a well-defined countably additive measure. Moreover: 
\begin{enumerate}
\item[(i)] $\nu$ is non-atomic and $\mathcal{R}(\nu)$ is relatively norm compact.
\item[(ii)] $L_1(\nu)$ is WASQ.
\item[(iii)] For each $n\in\N$, let $\lambda_n$ be the restriction of $\lambda$ to the $\sigma$-algebra on~$A_n$ given by
$\Sigma_n:=\{A\cap A_n:A\in \Sigma\}$. If $(p_n)_{n\in \N}$ is bounded, then the map 
$$
	\Phi: L_1(\nu) \to \ell_{(p_n)}(L_1(\lambda_n))
$$ 
given by
$$
	\Phi(f):=\big(f|_{A_n}\big)_{n\in \N}
	\quad\text{for all $f\in L_1(\nu)$}
$$
is a well-defined lattice isometry.
\end{enumerate}
\end{proposition}
\begin{proof}
Define $\tilde{\nu}:\Sigma \to \ell_1$ by
$$
	\tilde{\nu}(A):=\big(\lambda(A\cap A_n)\big)_{n\in \N}
	\quad
	\text{for all $A\in \Sigma$}.
$$
Note that $\tilde{\nu}$ is finitely additive and satisfies $\|\tilde{\nu}(A)\|_{\ell_1}=\lambda(A)$ for all $A\in \Sigma$,
hence $\tilde{\nu}$ is countably additive. Since the inclusion $\iota:\ell_1\hookrightarrow \ell_{(p_n)}$ is a well-defined operator,  the composition
$\nu=\iota\circ \tilde{\nu}:\Sigma \to \ell_{(p_n)}$ is a countably additive measure.

(i) Clearly, we have $\mathcal{N}(\lambda)=\mathcal{N}(\nu)$, so
$\nu$ is non-atomic. The range of any countably additive Banach space-valued measure
is relatively weakly compact (see, e.g., \cite[p.~14, Corollary~7]{die-uhl-J}). Hence, by the Schur property of~$\ell_1$,
the set $\mathcal{R}(\tilde{\nu})$ is relatively norm compact. Alternatively, this can also be deduced  
from the usual criterion of relative norm compactness in~$\ell_1$ (see, e.g., \cite[p.~6, Exercise~6]{die-J}).
Therefore, $\mathcal{R}(\nu)=\iota(\mathcal{R}(\tilde{\nu}))$ is relatively norm compact as well. 
 
(ii) follows from~(i) and Theorem~\ref{theorem:main} (note that $\nu$ takes values in $\ell_{(p_n)}^+$).

(iii) Fix $f\in L_1(\nu)$. For each $n\in \N$, let $\pi_n\in \ell_{(p_n)}^*$ be the the $n$th-coordinate functional. Since 
$(\pi_n\circ \nu)(A)=\lambda(A\cap A_n)$ for all $A\in \Sigma$ and $f\in L_1(\pi_n\circ \nu)$, we have $f|_{A_n} \in L_1(\lambda_n)$ and 
$$
	\pi_n\left( I_\nu(|f|) \right)=
	\int_{[0,1]} |f| \, d(\pi_n\circ \nu)=\left\|f|_{A_n}\right\|_{L_1(\lambda_n)}.
$$
Hence, $(\left\|f|_{A_n}\right\|_{L_1(\lambda_n)})_{n\in \N}=I_\nu(|f|)\in \ell_{(p_n)}$. Moreover, the fact that
$\nu$ takes values in $\ell_{(p_n)}^+$ ensures that
$$
	\|f\|_{L_1(\nu)}=\big\|I_\nu(|f|)\big\|_{\ell_{(p_n)}}=
	\left\|\left(\left\|f|_{A_n}\right\|_{L_1(\lambda_n)}\right)_{n\in \N}\right\|_{\ell_{(p_n)}}
$$
(see, e.g., \cite[Lemma~3.13]{oka-alt}). Thus, $\Phi$ is a well-defined isometric embedding. Clearly, $\Phi$ is a lattice homomorphism. It remains
to check that $\Phi$ is surjective.

Let $(f_n)_{n\in \N}\in \ell_{(p_n)}(L_1(\lambda_n))$. Define $f\in L_0[0,1]$ by declaring $f|_{A_n}:=f_n$ for all $n\in \N$. 
Since $(p_n)_{n\in \N}$ is bounded, the space $\ell_{(p_n)}$ contains no isomorphic copy of~$c_0$
(see, e.g., \cite[Theorem~3.5]{woo}). Therefore, in order to prove that $f\in L_1(\nu)$ it suffices to show
that $f\in L_1(|\varphi \circ \nu|)$ for every $\varphi\in \ell_{(p_n)}^*$ (see, e.g., \cite[p.~31, Theorem~1]{klu-kno}).
It is known that $\ell_{(p_n)}^*=\ell_{(q_n)}$, where $(q_n)_{n\in \N}$ is the sequence in~$(1,\infty)$
defined by $1/p_n+1/q_n=1$ for all $n\in \N$, the duality being
$$
	\big\langle (a_n)_{n\in \N},(b_n)_{n\in \N}\big\rangle=\sum_{n\in \N}a_nb_n
	\quad\text{for all $(a_n)_{n\in \N}\in \ell_{(p_n)}$ and $(b_n)_{n\in \N}\in \ell_{(q_n)}$}
$$
(see, e.g., \cite[Theorem~4.2]{woo}). Take any $\varphi=(b_n)_{n\in \N} \in \ell_{(q_n)}$. Then
$$
	(\varphi \circ \nu)(A)=\sum_{n\in \N}b_n \lambda(A\cap A_n)
	\quad
	\text{for all $A\in \Sigma$}
$$
and so the variation of $\varphi \circ \nu$ is given by
$$
	|\varphi \circ \nu|(A)=\sum_{n\in \N}|b_n| \lambda(A\cap A_n)
	\quad
	\text{for all $A\in \Sigma$}.
$$
Then
\begin{multline*}
	\int_{[0,1]} |f| \, d|\varphi \circ \nu|=
	\sum_{n\in \N} \int_{A_n} |f| \, d|\varphi \circ \nu| \\ =
	\sum_{n\in \N}|b_n|\int_{A_n}|f| \, d\lambda=
	\sum_{n\in \N}|b_n| \|f_n\|_{L_1(\lambda_n)}<\infty,
\end{multline*}
because $(\|f_n\big\|_{L_1(\lambda_n)})_{n\in \N} \in \ell_{(p_n)}$ and $\varphi\in \ell_{(q_n)}$.
Thus, $f\in L_1(\nu)$ and we have $\Phi(f)=(f_n)_{n\in \N}$. The proof is finished. 
\end{proof}

\begin{remark}
Note that each $L_1(\lambda_n)$ is WASQ (in fact, it is isometrically isomorphic to~$L_1[0,1]$). Hence, when $(p_n)_{n\in \N}$ is bounded,
the fact that $L_1(\nu)$ is WASQ can also be deduced from Corollary~\ref{corollary:Nakano-sum} and Proposition~\ref{proposition:example-L1nu}(iii).
\end{remark}

\begin{proposition}\label{proposition:example-L1nu-2}
Let $\nu$ be as in Proposition~\ref{proposition:example-L1nu}. 
If $(p_n)_{n\in \N}$ is bounded and $\frac{p_n}{(p_n-1)\log n}\to 0$ as $n\to \infty$, then $L_1(\nu)$ is not an $\mathcal{L}_1$-space.
\end{proposition}
\begin{proof}
Since $(p_n)_{n\in \N}$ is bounded, the space $\ell_{(p_n)}$ has an unconditional basis
(see, e.g., \cite[Theorem~3.5]{woo}). The additional condition on $(p_n)_{n\in \N}$ 
implies that $\ell_{(p_n)}$ is not isomorphic to~$\ell_1$, 
see \cite[Lemma~4]{wnu}. Therefore, $\ell_{(p_n)}$ cannot be isomorphic to
a complemented subspace of an $\mathcal{L}_1$-space (see, e.g., \cite[Theorem~3.13]{die-alt}).

Since $L_1(\nu)$ contains a complemented subspace isomorphic to~$\ell_{(p_n)}$ (this can be deduced from Proposition~\ref{proposition:example-L1nu}(iii)), 
it follows that $L_1(\nu)$ is not an $\mathcal{L}_1$-space.
\end{proof}

For instance, the sequence $p_n:=1+(\log (n+1))^{-1/2}$ satisfies the conditions of Proposition~\ref{proposition:example-L1nu-2}.

\section{Proof of Theorem~\ref{theorem:countere}}\label{section:renorming}

The aim of this section is to provide a proof of Theorem \ref{theorem:countere}.
The first step is to prove the result for the space~$c_0$, see Theorem~\ref{theorem:countpartcas} below.
The proof of this particular case is based on the renorming technique of \cite[Theorem~2.4]{bec-lop-rue:15a},
where it was shown that every Banach space containing an isomorphic copy of~$c_0$
admits an equivalent norm so that its unit ball contains non-empty relatively weakly open subsets with arbitrarily small diameter, but every slice has diameter~$2$.

The symbol $\mathbb N^{<\omega}$ stands for the {\em Baire tree}, i.e., the set of all {\em finite} 
sequences of positive integers. The empty sequence is included in $\mathbb N^{<\omega}$
as the root of the tree. The order on $\mathbb N^{<\omega}$ is defined by declaring that $\alpha \preceq \beta$ if and only if $\beta$ extends $\alpha$. 
Given $\alpha\in \N^{<\omega}$ and $p\in \N$, we denote by $\alpha\smallfrown p \in \N^{<\omega}$ the sequence defined
by $\alpha\smallfrown p:=(\alpha_1,\dots,\alpha_n,p)$ if $\alpha=(\alpha_1,\dots,\alpha_n)$ or
$\alpha\smallfrown p:=(p)$ (a sequence with just one element) if $\alpha=\emptyset$. The following is standard (see, e.g., \cite[p.~857]{bec-lop-rue:15a}):

\begin{lemma}\label{lemma:PHI}
There exists a bijection $\phi:\mathbb N^{<\omega}\rightarrow \mathbb N$ such that:
\begin{enumerate}
\item[(i)] $\phi (\emptyset)=1$.
\item[(ii)] $\phi(\alpha)\leq\phi(\beta)$ for all $\alpha,\beta \in \mathbb N^{<\omega}$ with $\alpha\preceq \beta$.
\item[(iii)] $\phi(\alpha\smallfrown j)< \phi(\alpha \smallfrown k)$ for every $\alpha\in \mathbb N^{<\omega}$ and for all $j, k\in\mathbb N$ with $j< k$.
\end{enumerate}
\end{lemma}
  
Let $c$ be the subspace of~$\ell_\infty$ consisting of all convergent sequences and let
$c(\mathbb N^{<\omega})$ be the subspace of~$\ell_\infty(\mathbb N^{<\omega})$ defined by
$$
	c(\mathbb N^{<\omega}):=\{x\in \ell_\infty(\mathbb N^{<\omega}): \, x\circ \phi^{-1} \in c\}.
$$
Clearly, $c(\mathbb N^{<\omega})$ and $c$ are isometric, hence $c(\mathbb N^{<\omega})$ is isomorphic to~$c_0$.
We denote by $\lim \in c(\mathbb N^{<\omega})^*$ the functional defined by
$$
	\lim x:=\lim_{n\to \infty} x(\phi^{-1}(n))
	\quad
	\text{for all $x\in c(\mathbb N^{<\omega})$}.
$$
For each $\alpha\in \mathbb N^{<\omega}$ we denote by $e_\alpha^*\in c(\mathbb N^{<\omega})^*$
the functional defined by
$$
	e_\alpha^*(x):=x(\alpha) 	\quad
	\text{for all $x\in c(\mathbb N^{<\omega})$}.
$$
Given $\alpha\in \mathbb N^{<\omega}$ we define $x_\alpha\in \ell_\infty(\mathbb N^{<\omega})$ by the formula
$$
	x_\alpha(\beta):=\left\{
	\begin{array}{cc}
	1 & \mbox{if }\beta\preceq \alpha\\
	-1 & \mbox{otherwise}
	\end{array} \right.
$$
so that $x_\alpha\in S_{c(\mathbb N^{<\omega})}$ and $\lim x_\alpha=-1$. Define
$$
	A:=\{x_\alpha: \, \alpha\in\mathbb N^{<\omega}\} \sub S_{c(\N^{<\omega})}
	\quad\mbox{and} \quad
	K:=\overline{\conv}(A\cup -A) \sub B_{c(\N^{<\omega})}.
$$

We will need the following result (see \cite[Proposition~2.2]{bec-lop-rue:15a}):

\begin{lemma}\label{lemma:smallopenweak}
Let $n\in\mathbb N$ and $\rho>0$. Define
\begin{multline*}
	W_{n,\rho}:=\Big\{x\in K:  \, e_{\emptyset\smallfrown i}^*(x)>\frac{2}{n}-1-2\rho \ \text{ for all }i\in \{1,\dots,n\}  \\ \text{ and }\lim x <-1+\rho\Big\}.
\end{multline*}
Then $W_{n,\rho}$ is a non-empty relatively weakly open subset of~$K$ and ${\rm diam}(W_{n,\rho})\to 0$ as $n\to \infty$ and $\rho\to 0$.
\end{lemma}

The following lemma is elementary and its proof will be omitted:

\begin{lemma}\label{lemma:convex-decomposition}
Let $V$ be a linear space, let $A_1,\dots,A_m$ be subsets of~$V$ and let $v\in {\rm conv}(A_1\cup \dots\cup  A_m)$.
Then there exist $v_i\in {\rm conv}(A_i)$ and $c_i \in [0,\infty)$ for $i\in \{1,\dots,m\}$
such that $\sum_{i=1}^m c_i v_i=v$ and $\sum_{i=1}^m c_i=1$.
\end{lemma}

\begin{lemma}\label{lemma:SQ}
Let $(X,\|\cdot\|)$ be a Banach space, let $S \sub S_{(X,\|\cdot\|)}$ be a dense set and let $0<\delta<1$.
Suppose that for all $0<r,s < \delta$ and for every finite set $\{x_1,\ldots, x_n\}\sub S$ there exists $y\in S_{(X,\|\cdot\|)}$ satisfying
$$
	\|rx_i \pm sy\|\le1 \quad\text{for every $i\in \{1,\dots,n\}$.}
$$ 
Then $(X,\|\cdot\|)$ is $(r,s)$-SQ for all $0<r,s<\delta$.
\end{lemma}
\begin{proof}
Fix $0<r,s<\delta$. Choose $r<r'<\delta$ such that $s':=s\frac{r'}{r}<\delta$ and then
choose $\theta>0$ such that $r\theta+\frac{r}{r'}\leq 1$. 
Take any finite set $\{x_1,\ldots, x_n\}\sub S_{(X,\|\cdot\|)}$. Since $S$ is dense in~$S_{(X,\|\cdot\|)}$, there exist $x'_1,\dots,x'_n\in S$
such that $\|x_i-x'_i\|\leq \theta$ for every $i\in \{1,\dots,n\}$. By the assumption, we can find $y\in S_{(X,\|\cdot\|)}$
in such a way that $\|r'x'_i \pm s'y\|\le1$ for every $i\in \{1,\dots,n\}$.
Then
$$
	\|rx_i \pm sy\|\leq
	r\|x_i-x'_i\|+\frac{r}{r'}\|r'x'_i \pm s'y\| \leq r\theta +  \frac{r}{r'} \leq 1
$$
for every $i\in \{1,\dots,n\}$. This shows that $(X,\|\cdot\|)$ is $(r,s)$-SQ.
\end{proof}

\begin{theorem}\label{theorem:countpartcas}
Let $0<\varepsilon<1$. Then there exists an equivalent norm $|\cdot|$ on~$c_0$ such that:
\begin{enumerate}
\item[(i)] $(c_0,|\cdot|)$ has the slice-D2P.
\item[(ii)] There are non-empty relatively weakly open subsets of $B_{(c_0,|\cdot|)}$ of arbitrarily small diameter.
\item[(iii)] $(c_0,|\cdot|)$ is $(r,s)$-SQ for all $0<r,s < \frac{1-\varepsilon}{1+\varepsilon}$.
\end{enumerate}
\end{theorem}
\begin{proof}
Let us denote by $\|\cdot\|_Z$ the norm of the Banach space $Z:=c(\mathbb N^{<\omega})\oplus_\infty c_0$. Since $c_0$ and $Z$ are isomorphic, 
it suffices to prove the statement of the theorem for the space~$(Z,\|\cdot\|_Z)$.
Let $c_0(\mathbb N^{<\omega}) \sub c(\mathbb N^{<\omega})$ be the subspace of all $x\in c(\mathbb N^{<\omega})$
such that $\lim x=0$. For each $\alpha\in \mathbb N^{<\omega}$ we write $e_\alpha$ to denote the element of $S_{c_0(\mathbb N^{<\omega})}$
defined by $e_\alpha(\alpha)=1$ and $e_\alpha(\beta)=0$ for every $\beta \in \mathbb N^{<\omega}\setminus \{\alpha\}$.

Let $|\cdot|$ be the Minkowski functional of the closed convex symmetric set
$$
	B:=\overline{\conv}\Big((A\times\{0\})\cup ({-A}\times \{0\})\cup \big((1-\varepsilon)B_Z+\varepsilon B_{c_0(\mathbb N^{<\omega})}\times\{0\}\big)\Big) \sub Z,
$$
that is, $|z|:=\inf\{t>0: z\in tB\}$ for all $z\in Z$. Since $(1-\varepsilon)B_Z\subseteq B\subseteq B_Z$, it follows that $|\cdot|$ is an equivalent norm on~$Z$ with
unit ball $B_{(Z,|\cdot|)}=B$. We have
\begin{equation}\label{equation:equivalence-norms}
	\|z\|_Z \leq |z|\leq \frac{1}{1-\varepsilon}\|z\|_Z\quad\text{for all $z\in Z$}
\end{equation}
and
\begin{equation}\label{equation:coincidence-norms}
	\vert (x,0)\vert=\Vert (x,0)\Vert_Z=\Vert x\Vert_\infty
	\quad
	\mbox{for all }x\in c_0(\mathbb N^{<\omega}), 
\end{equation}
because $B_{c_0(\mathbb N^{<\omega})}\times \{0\}\subseteq B$. We denote by $\|\cdot\|_{Z^*}$ and $|\cdot|_{Z^*}$
the equivalent norms on~$Z^*$ induced by $\|\cdot\|_Z$ and $|\cdot|$, respectively. We will check that
$|\cdot|$ satisfies the required properties.

\smallskip
{\em Proof of~(i).} Let $S \sub B$ be a slice of~$B$. Since $B \setminus S$ is convex and closed, we have
$$
	S\cap \Big((A\times\{0\})\cup ({-A}\times \{0\})\cup \big((1-\varepsilon)B_Z+\varepsilon B_{c_0(\mathbb N^{<\omega})}\times\{0\}\big)\Big)\neq\emptyset.
$$
We now distinguish several cases.

Case (a): $S\cap (A\times \{0\})\neq \emptyset$. Then $(x_\alpha,0)\in S$ for some $\alpha\in \mathbb N^{<\omega}$. 
Observe that the sequence $((x_{\alpha\smallfrown n},0))_{n\in \N}=((x_\alpha+2e_{\alpha\smallfrown n},0))_{n\in \N}$ converges weakly to $(x_\alpha,0)$ in~$Z$
(because $(e_{\alpha\smallfrown n})_{n\in \N}$ is weakly null in~$c_0(\mathbb N^{<\omega})$).
Since $S$ is relatively weakly open in~$B$ and $(x_{\alpha\smallfrown n},0)\in B$ for every $n\in\mathbb N$, we have 
$(x_{\alpha\smallfrown n_0},0)\in S$ for large enough $n_0\in \N$. Hence,
$$
	\diam_{|\cdot|}(S)\geq \vert (x_{\alpha\smallfrown n_0},0)-(x_{\alpha},0)\vert=2\vert (e_{\alpha\smallfrown n_0},0)\vert
	\stackrel{\stackrel{\eqref{equation:coincidence-norms}}{}}{=}2\Vert (e_{\alpha\smallfrown n_0},0)\Vert_Z=2
$$
and, therefore, $\diam_{|\cdot|}(S)=2$.

Case (b): $S\cap (-A\times \{0\})\neq \emptyset$. The proof that $\diam_{|\cdot|}(S)=2$ runs similarly as in~(a).

Case (c): $S\cap ((1-\varepsilon)B_Z+\varepsilon B_{c_0(\mathbb N^{<\omega})}\times \{0\})\neq \emptyset$. Then we can pick
$(x,y)\in B_Z$ and $x'\in B_{c_0(\mathbb N^{<\omega})}$ in such a way that 
$$
	z:=(1-\varepsilon)(x,y)+\varepsilon(x',0) \in S.
$$ 
We can assume without loss of generality that $x'$ has finite support, because the set of all finitely supported functions from $\mathbb N^{<\omega}$ to~$[-1,1]$ is 
a norm dense subset of~$B_{c_0(\mathbb N^{<\omega})}$. Choose $\alpha\in \mathbb N^{<\omega}$ such that $x'(\alpha\smallfrown n)=0$ 
for every $n\in\mathbb N$. Observe that
$$
	x-x(\alpha\smallfrown n)e_{\alpha\smallfrown n}\pm e_{\alpha\smallfrown n}\in B_{c(\mathbb N^{<\omega})}
	\quad
	\mbox{and}
	\quad
	x'\pm e_{\alpha\smallfrown n}\in B_{c_0(\mathbb N^{<\omega})}
$$
and so 
\begin{multline*}
	z_n^{\pm}:=(1-\varepsilon)\big(x-x(\alpha\smallfrown n)e_{\alpha\smallfrown n},y\big)+\varepsilon (x',0)\pm (e_{\alpha\smallfrown n},0) \\
	= (1-\varepsilon)\big(x-x(\alpha\smallfrown n)e_{\alpha\smallfrown n}\pm e_{\alpha\smallfrown n},y\big)+
	\varepsilon (x'\pm e_{\alpha\smallfrown n},0)\in B
\end{multline*}
for every $n\in\mathbb N$. Since $S$ is relatively weakly open in~$B$ and both sequences $(z_n^{+})_{n\in \N}$ and $(z_n^{-})_{n\in \N}$ converge weakly to~$z$ in~$Z$,
we can find $n_0\in\mathbb N$ large enough so that both $z_{n_0}^{+}$ and $z_{n_0}^{-}$ belong to~$S$. Hence,
$$
	\diam_{|\cdot|}(S)\geq |z_{n_0}^{+} - z_{n_0}^{-}| =2\vert (e_{\alpha\smallfrown n_0},0)\vert \stackrel{\stackrel{\eqref{equation:coincidence-norms}}{}}{=}  
	2\Vert (e_{\alpha\smallfrown n_0},0)\Vert_Z=2
$$
and so $\diam_{|\cdot|}(S)=2$. This finishes the proof of~(i).

\smallskip
{\em Proof of~(ii).} Fix $\theta>0$. By Lemma~\ref{lemma:smallopenweak}, we can take $n\in \N$ and $\rho>0$ such that
\begin{equation}\label{equation:choice-n-rho}
	\diam(W_{n,\rho})\leq \frac{(1-\varepsilon)\theta}{2}
	\quad\mbox{and}\quad
	 \eta:=\frac{2\rho}{5} \leq \frac{\theta}{16}.
\end{equation}
Define 
\begin{multline*}
	U:=\Big\{z\in B:  \, (e_{\emptyset\smallfrown i}^*,0)(z)>\frac{2}{n}-1-\eta \, \text{ for all }i\in \{1,\dots,n\}  \\
	 \text{and } (\lim,0)(z)<-1+\varepsilon\eta\Big\}.
\end{multline*}
It is clear that $U$ is a relatively weakly open subset of~$B$. To prove that $U\neq \emptyset$ we will check that 
the vector $z_0:=(\frac{1}{n}\sum_{j=1}^n x_{\emptyset\smallfrown j},0)\in B$ belongs to~$U$. Indeed,
for each $i\in \{1,\dots,n\}$ we have
$$
	(e_{\emptyset\smallfrown i}^*,0)(z_0)=\frac{1}{n}
	\sum_{j=1}^n e_{\emptyset\smallfrown i}^*(x_{\emptyset\smallfrown j})=\frac{1}{n}\big(1-(n-1)\big)=\frac{2}{n}-1>\frac{2}{n}-1-\eta
$$ 
and we also have
$$
	(\lim,0)(z_0)=\frac{1}{n}\sum_{j=1}^n \lim x_{\emptyset\smallfrown j}=-1<-1+\varepsilon\eta.
$$
Hence $z_0\in U$ and so $U\neq \emptyset$.

We will show that $\diam_{|\cdot|}(U)\leq \theta$. The key point is the following:

{\em Claim. For every 
$$
	z\in V:=U \cap {\rm conv}\Big((A\times\{0\})\cup ({-A}\times \{0\})\cup \big((1-\varepsilon)B_Z+\varepsilon B_{c_0(\mathbb N^{<\omega})}\times\{0\}\big)\Big)
$$
there is $z'\in W_{n,\rho}\times \{0\}$ such that $|z-z'|<4\eta$.} 

Indeed, by Lemma~\ref{lemma:convex-decomposition}
we can write 
\begin{equation}\label{equation:z}
	z=a z_1-b z_2+c((1-\varepsilon)u+\varepsilon v)
\end{equation}
for some $a,b,c\geq 0$ with $a+b+c=1$ and
$$
	z_1,z_2\in {\rm conv}(A\times \{0\}), \
	u\in B_{Z}\ \mbox{and} \
	v\in B_{c_0(\mathbb N^{<\omega})}\times\{0\}.
$$
Observe that 
$$
	(\lim,0)(z_1)=(\lim,0)(z_2)=-1,
$$ 
because $\lim x_\alpha=-1$ for all $\alpha \in \mathbb N^{<\omega}$. We also have
$$
	\vert(\lim,0)(u)\vert\leq \|(\lim,0)\|_{Z^*} \|u\|_Z \leq 1
	\quad\mbox{and}\quad 
	(\lim,0)(v)=0.
$$ 
Thus
\begin{multline*}
	-1+\varepsilon\eta> (\lim,0)(z) \stackrel{\stackrel{\eqref{equation:z}}{}}{=} -a+b+c(1-\varepsilon)(\lim ,0)(u)  \\ \geq -a+b-c(1-\varepsilon) =-a-b-c+2b+\varepsilon c=
	-1+2b+\varepsilon c,
\end{multline*}
and so $2b+\varepsilon c<\varepsilon\eta$. This inequality implies that $b< \eta$ (bear in mind that $\epsilon<1$) and that $c<\eta$. 
Consequently
\begin{eqnarray*}
	\vert z-z_1\vert &\stackrel{\stackrel{\eqref{equation:z}}{}}{=}& \Big| (a-1) z_1-b z_2+c((1-\varepsilon)u+\varepsilon v) \Big| \\
	& = & \Big| -b(z_1+z_2)+c((1-\varepsilon)u+\varepsilon v-z_1) \Big| \\
	& \leq & b |z_1|+b|z_2|+c\big|(1-\varepsilon)u+\varepsilon v\big|+c|z_1| \\
	& \stackrel{(\star)}{\leq} & 2b+2c<4\eta,
\end{eqnarray*}
where inequality $(\star)$ follows from the fact that $z_1$, $z_2$ and $(1-\varepsilon)u+\varepsilon v$ belong to~$B=B_{(Z,|\cdot|)}$. 
Hence, $|z-z_1|<4\eta$. 

We can write $z_1=(x,0)$ for some $x\in {\rm conv}(A)$. Then $\lim x=-1$ and for each $i\in \{1,\dots,n\}$ we have
$$
	e_{\emptyset\smallfrown i}^*(x)=
	(e_{\emptyset\smallfrown i}^*,0)(z_1)\geq (e_{\emptyset\smallfrown i}^*,0)(z)-\vert (e_{\emptyset\smallfrown i}^*,0)\vert_{Z^*} \vert z-z_1\vert
	>
	\frac{2}{n}-1-5\eta,
$$
because $z\in U$ and $\vert (e_{\emptyset\smallfrown i}^*,0)\vert_{Z^*} \leq \|(e_{\emptyset\smallfrown i}^*,0)\|_{Z^*}=1$
(by~\eqref{equation:equivalence-norms}). This implies, with the notation of Lemma~\ref{lemma:smallopenweak}, that 
$x\in W_{n,\rho}$ (recall that $\eta=\frac{2}{5}\rho$). 
Therefore, the conclusion
of the {\em Claim} holds taking $z'=z_1$.

Finally, let $w_1,w_2\in U$ and fix $s>0$. Since $U$ is relatively open in~$B$, we
can find $v_1,v_2\in V$ such that $|w_1-v_1|\leq s$ and $|w_2-v_2|\leq s$. By the {\em Claim} above,
there exist $v'_1,v'_2\in W_{n,\rho}\times \{0\}$ such that $|v_1-v'_1|<4\eta$ and $|v_2-v'_2|<4\eta$. Then 
$$
	|v'_1-v'_2| \stackrel{\stackrel{\eqref{equation:equivalence-norms}}{}}{\leq}
	\frac{1}{1-\varepsilon}\|v'_1-v'_2\|_Z \stackrel{\stackrel{\eqref{equation:coincidence-norms}}{}}{\leq} 
	\frac{1}{1-\varepsilon}\diam(W_{n,\rho}) \stackrel{\stackrel{\eqref{equation:choice-n-rho}}{}}{\leq} \frac{\theta}{2}
$$
and so
$$
	|w_1-w_2| < 2s+8\eta + \frac{\theta}{2}\stackrel{\stackrel{\eqref{equation:choice-n-rho}}{}}{\leq} 2s+\theta.
$$
As $w_1,w_2\in U$ and $s>0$ are arbitrary, we conclude that $\diam_{|\cdot|}(U)\leq \theta$.

\smallskip
{\em Proof of~(iii).} We will show that $(Z,|\cdot|)$ is $(r,s)$-SQ for any $0<r,s < \frac{1-\varepsilon}{1+\varepsilon}$
with the help of Lemma~\ref{lemma:SQ}. Let $H$ be the norm dense subset of $B_{c_{0}}$
consisting of all finitely supported functions from $\N$ to~$[-1,1]$. Then the set
$$
	S:=S_{(Z,|\cdot|)} \cap {\rm conv}\Big((A\times\{0\})\cup ({-A}\times \{0\})\cup \big((1-\varepsilon)(B_{c(\N^{\omega})} \times H)+
	\varepsilon B_{c_0(\N^{\omega})} \times\{0\}\big)\Big)
$$ 
is norm dense in $S_{(Z,|\cdot|)}$. Fix $0<r,s < \frac{1-\varepsilon}{1+\varepsilon}$ and 
take finitely many $z_1,\ldots, z_m\in S$. By Lemma~\ref{lemma:convex-decomposition}, each $z_i$ can be written as
$$
	z_i=a_i (x^1_i,0)+b_i (-x^2_i,0)+c_i\big((1-\varepsilon)(x_i,y_i)+\varepsilon (x^3_i,0)\big)
$$ 
for some $a_i,b_i,c_i\geq 0$ with $a_i+b_i+c_i=1$ and
$$
	x^1_i,x^2_i\in {\rm conv}(A), \
	x_i \in B_{c(\N^{\omega})},\ y_i\in H \ \mbox{and} \
	x_i^3 \in B_{c_0(\mathbb N^{<\omega})}.
$$
Let $(e_n)_{n\in \N}$ be the usual basis of~$c_0$ and choose $n\in\mathbb N$ large enough 
such that $\Vert y_i\pm e_n\Vert_{\infty}\leq 1$ for all $i\in \{1,\dots,m\}$. 

Observe that for each $i\in \{1,\dots,m\}$ we have
$$
	z_i\pm (0,e_n)=a_i (x^1_i,\pm e_n)+b_i (-x^2_i, \pm e_n)+c_i\big((1-\varepsilon)(x_i,y_i\pm e_n)+\varepsilon (x^3_i,\pm e_n)\big),
$$
thus
\begin{multline*}
	\| z_i\pm (0,e_n)\|_Z \\ \leq a_i\|(x^1_i,\pm e_n)\|_Z+b_i \|(x^2_i,\pm e_n)\|_Z+c_i\big((1-\varepsilon)\|(x_i,y_i\pm e_n)\|_Z+\varepsilon \|(x^3_i,\pm e_n)\|_Z\big)
 	\\ \leq a_i+b_i+c_i((1-\varepsilon)+\varepsilon)) = 1,
\end{multline*}
which combined with~\eqref{equation:equivalence-norms} yields
\begin{equation}\label{equation:not-normalized}
	\vert z_i\pm (0,e_n)\vert\leq \frac{1}{1-\varepsilon}.
\end{equation}
Another appeal to~\eqref{equation:equivalence-norms} gives $1 =\|(0,e_n)\|_Z\leq |(0,e_n)|\leq \frac{1}{1-\varepsilon}$ and, therefore, 
the vector $z:=\frac{1}{|(0,e_n)|}(0,e_n)\in S_{(Z,|\cdot|)}$ satisfies
$$
	|(0,e_n)-z|=\left|\left(1-\frac{1}{|(0,e_n)|}\right)(0,e_n) \right|=
	|(0,e_n)|-1 \leq \frac{1}{1-\varepsilon}-1 =\frac{\varepsilon}{1-\varepsilon}.
$$
This inequality and~\eqref{equation:not-normalized} give
$$
	\vert z_i\pm z\vert \leq \frac{1}{1-\varepsilon}+\frac{\varepsilon}{1-\varepsilon}=\frac{1+\varepsilon}{1-\varepsilon}
	\quad\text{for every $i\in \{1,\dots,m\}$}.
$$
Since $r,s \leq  \frac{1-\varepsilon}{1+\varepsilon}$, we can apply \cite[Lemma 6.3]{avi-alt-7} to conclude that
$$
	\vert rz_i\pm s z\vert \leq 1 \quad\text{for every $i\in \{1,\dots,m\}$}.
$$
From Lemma~\ref{lemma:SQ} it follows that $(Z,|\cdot|)$ is $(r,s)$-SQ for all $0<r,s<\frac{1-\varepsilon}{1+\varepsilon}$.
The proof is finished.
\end{proof}

We can now prove Theorem~\ref{theorem:countere} in full generality.

\begin{proof}[Proof of Theorem \ref{theorem:countere}]
Let $Z$ and $W$ be subspaces of~$X$ such that $Z$ is isomorphic to~$c_0$ and $X=Z\oplus W$.
Fix $0<\varepsilon<1$ and let $|\cdot|_Z$ be an equivalent norm on~$Z$ like in Theorem~\ref{theorem:countpartcas}. 
Since any Banach space admits an equivalent norm for which the unit ball has slices of arbitrarily small diameter (see, e.g., \cite[Lemma~2.1]{bec-lop-rue:16b}),
we can take an equivalent norm $|\cdot|_W$ on~$W$ satisfying that property. 
Let $|\cdot|$ be the equivalent norm on~$X$ defined by $|z+w|:=\max\{|z|_Z,|w|_W\}$ for every $z\in Z$ and for every $w\in W$.
We claim that $(X,|\cdot|)$ satisfies all the requirements. 

(i) The $\ell_\infty$-sum of two Banach spaces has the slice-D2P whenever one of the factors has the slice-D2P (see, e.g., \cite[Theorem~2.29]{lan-thesis}).
Since $(Z,|\cdot|_Z)$ has the slice-D2P, we conclude that the same holds for $(X,|\cdot|)$.

(ii) Let us prove that $B_{(X,|\cdot|)}$ contains non-empty relatively weakly open subsets of arbitrarily small diameter. Fix $\eta>0$. Then there exists a non-empty 
relatively weakly open set $U \sub B_{(Z,|\cdot|_Z)}$ with $\diam_{|\cdot|_Z}(U)<\eta$. Now, take a slice $S$ of $B_{(W,|\cdot|_W)}$ 
with $\diam_{|\cdot|_W}(S)<\eta$. Since the map $\varphi:B_{(Z,|\cdot|_Z)}\times B_{(W,|\cdot|_W)} \to B_{(X,|\cdot|)}$ given by
$$
	\varphi(z,w):=z+w 
	\quad\text{for all $(z,w)\in B_{(Z,|\cdot|_Z)}\times B_{(W,|\cdot|_W)}$}
$$
is a homeomorphism when each of the balls is equipped with the restriction of the weak topology, it follows
that $V:=\varphi(U\times S)$ is a relatively weakly open subset of $B_{(X,|\cdot|)}$. Clearly, $V\neq \emptyset$
and $\diam_{|\cdot|}(V)<\eta$, as desired.

(iii) The space $(X,|\cdot|)$ is $(r,s)$-SQ for arbitrary $0<r,s < \frac{1-\varepsilon}{1+\varepsilon}$, because so is $(Z,|\cdot|_Z)$ and 
the $\ell_\infty$-sum of two Banach spaces is $(r,s)$-SQ whenever one of the factors is $(r,s)$-SQ (see \cite[Proposition~6.6]{avi-alt-7}).
\end{proof}

\subsection*{Acknowledgements} 
We thank Antonio Avil\'{e}s for valuable discussions on the results of this paper.
The research was supported by grants PID2021-122126NB-C32 (J. Rodr\'{i}guez) and 
PID2021-122126NB-C31 (A. Rueda Zoca), funded by \break MCIN/AEI/10.13039/501100011033 and ``ERDF A way of making Europe'', 
and also by grant 21955/PI/22 (funded by Fundaci\'on S\'eneca - ACyT Regi\'{o}n de Murcia). 
The research of A. Rueda Zoca was also supported by grants FQM-0185 and PY20\_00255 (funded by Junta de Andaluc\'ia).


\bibliographystyle{amsplain}

\end{document}